\documentclass[a4paper,11pt]{article}

\usepackage[T1]{fontenc}
\usepackage{lmodern}

\usepackage{dsfont}

\usepackage[latin1]{inputenc}
\usepackage{amsmath}
\usepackage{amsthm}
\usepackage{amssymb}
\usepackage{mathrsfs}
\usepackage{graphicx}
\usepackage[all]{xy}
\usepackage{hyperref}

\usepackage{makeidx}

\usepackage{stmaryrd}
\usepackage{caption}

\usepackage{abstract} 

\newtheorem{thm}{Theorem}[section]
\newtheorem{cor}[thm]{Corollary}

\newtheorem{fact}[thm]{Fact}

\newtheorem{lemma}[thm]{Lemma}
\newtheorem{prop}[thm]{Proposition}

\theoremstyle{definition}
\newtheorem{definition}[thm]{Definition}
\newtheorem{ex}[thm]{Example}
\newtheorem{exercice}[thm]{Exercice}

\def\rquotient#1#2{%
	\makeatletter
	\raise.3ex\hbox{$#1$}/\lower.3ex\hbox{$#2$}%
	\makeatother
}	

\makeatletter
\newcommand{\subjclass}[2][2010]{%
	\let\@oldtitle\@title%
	\gdef\@title{\@oldtitle\footnotetext{#1 \emph{Mathematics subject classification.} #2}}%
}
\newcommand{\keywords}[1]{%
	\let\@@oldtitle\@title%
	\gdef\@title{\@@oldtitle\footnotetext{\emph{Key words and phrases.} #1.}}%
}
\makeatother

\newcommand{\Address}{{
		\bigskip
		\small
		
		\textsc{University of Montpellier\\ 
Institut Math\'ematiques Alexander Grothendieck\\
Place Eug\`ene Bataillon\\
34090 Montpellier (France)}\par\nopagebreak
		\textit{E-mail address}: \texttt{anthony.genevois@umontpellier.fr}
		
}}

\makeindex

\title{Mini-course on the embedding problem between RAAGs}
\date{CIRM, 11-15 April, 2022}
\author{Anthony Genevois}

\begin{document}

\maketitle

\begin{abstract}
Given a graph $\Gamma$, the \emph{right-angled Artin group} $A(\Gamma)$, also referred to as a \emph{partially commutative group}, is given by the presentation $\langle u \in V(\Gamma) \mid [u,v]=1, \ \{u,v\} \in E(\Gamma) \rangle$. This mini-course is dedicated to the Embedding Problem in right-angled Artin groups, namely: given two finite graphs $\Phi,\Psi$, how to determine whether or not $A(\Phi)$ is isomorphic to a subgroup of $A(\Psi)$? We describe a geometric framework, based on the geometry of \emph{quasi-median graphs}, in which this question can be studied.
\end{abstract}

\tableofcontents

\vspace{1cm}

\medskip \noindent
The Embedding Problem in right-angled Artin groups is still widely open, but a major progress has been recently made in \cite{MR3039768}. The results we present in Lecture~3 essentially come from this work, but our approach is completely different. While in \cite{MR3039768} mapping class groups play a central role as auxiliary objects, here we describe how the quasi-median geometry of right-angled Artin groups leads, in our opinion, to a more natural and conceptually simpler formulation of the results in \cite{MR3039768}. This point of view comes from \cite{QM}.

\newpage

\section{Lecture 1: A crash course on RAAGs}

\noindent
In this mini-course, we focus on a specific family of groups, namely \emph{right-angled Artin groups}, or \emph{RAAGs} for short. 

\begin{definition}\label{def:RAAG}
Let $\Gamma$ be a graph. The \emph{right-angled Artin group} $A(\Gamma)$ is the group given by the presentation
$$\langle u \in V(\Gamma) \mid [u,v]=1, \ \{u,v\} \in E(\Gamma) \rangle,$$
where $V(\Gamma)$ and $E(\Gamma)$ denote the vertex- and edge-sets of $\Gamma$. 
\end{definition}

\noindent
Usually, one says that right-angled Artin groups interpolate between free groups and free abelian groups. This is because, if $\Gamma$ has no edge, then there is no relation and $A(\Gamma)$ is a free group of rank $|V(\Gamma)|$; and, if $\Gamma$ is a complete graph, i.e. if any two vertices of $\Gamma$ are connected by an edge, then $A(\Gamma)$ is a free abelian group of rank $|V(\Gamma)|$. As other examples, observe that $A(\Gamma) \simeq \mathbb{F}_2 \oplus \mathbb{F}_2$ if $\Gamma$ is a $4$-cycle and $A(\Gamma) \simeq \mathbb{Z} \oplus \mathbb{F}_2$ if $\Gamma$ is a path of length two.

\medskip \noindent
To which extend does one understand right-angled Artin groups? In order to answer such a question, one can ask basic questions such as: When is a right-angled Artin group trivial? infinite? How to know whether two given elements commute? What are the possible finite subgroups of a given right-angled Artin group? and its possible solvable subgroups? How to know whether two given right-angled Artin groups are isomorphic? It turns out that all these questions are well understood. Let us give some details in this direction.

\paragraph{WORD PROBLEM.} Fix a graph $\Gamma$. Because the vertices of $\Gamma$ provide generators of $A(\Gamma)$, every element of $A(\Gamma)$ can be described as a word written over the alphabet $V(\Gamma) \sqcup V(\Gamma)^{-1}$. Clearly, the following operations do not modify the element of $A(\Gamma)$ represented by a given word $w:= u_1 \cdots u_n$. 
\begin{description}
	\item[(Free reduction)] If $u_{i+1}=u_i^{-1}$ or $u_i=u_{i+1}^{-1}$, remove the two letters from $w$.
	\item[(Shuffling)] If $u_i$ and $u_{i+1}$ are given by two vertices adjacent in $\Gamma$, replace $u_iu_{i+1}$ with $u_{i+1}u_i$ in $w$.
\end{description}
A word is \emph{graphically reduced}, or \emph{$\Gamma$-reduced} if one wants to specify $\Gamma$, if it cannot be shortened by applying shufflings and free reductions. 

\begin{prop}\label{prop:NormalForm}
Let $g \in A(\Gamma)$ be an element. A word has minimal length among all the words representing $g$ if and only if it is graphically reduced. Moreover, any two graphically reduced words representing $g$ differ only by finitely many shufflings.
\end{prop}

\noindent
See Exercice~\ref{ex:WP} for a proof. It follows that the word problem can be easily solved. Given a word in generators, apply shufflings and free reductions as much as possible. Since there exist only finitely many possible words of a given length, the process has to stop, given a graphically reduced word. The element represented by our initial word is then trivial if and only if the graphically reduced word thus obtained is empty.

\medskip \noindent
As an explicit example, assume that $\Gamma$ is a $5$-cycle. Index its vertices cyclically as $a,b,c,d,e$. We claim that $eade^{-1}c=bacb^{-1}d$. Indeed,
$$eade^{-1}c \overset{\text{shuffling}}{\longrightarrow} aede^{-1}c \overset{\text{shuffling}}{\longrightarrow} adee^{-1}c \underset{\text{reduction}}{\overset{\text{free}}{\longrightarrow}} adc$$
gives a graphically reduced word, and similarly
$$bacb^{-1}d \overset{\text{shuffling}}{\longrightarrow} abcb^{-1}d \overset{\text{shuffling}}{\longrightarrow} acbb^{-1}d  \underset{\text{reduction}}{\overset{\text{free}}{\longrightarrow}} acd$$
gives another graphically reduced word. But $adc$ can be transformed into $acd$ by a shuffling, hence the desired equality. 

\medskip \noindent
The (almost) normal form provided by Proposition~\ref{prop:NormalForm} is extremely useful. Let us mention a few easy applications. 

\begin{cor}\label{cor:SubGraph}
Let $\Gamma$ be a graph. For every induced subgraph $\Lambda$, the inclusion map $V(\Lambda) \to V(\Gamma)$ induces an injective morphism $A(\Lambda) \hookrightarrow A(\Gamma)$.
\end{cor}

\noindent
Recall that a subgraph $\Lambda \subset \Gamma$ is \emph{induced} if two vertices of $\Lambda$ are adjacent in $\Lambda$ if and only if they are adjacent in $\Gamma$. 

\begin{proof}[Proof of Corollary~\ref{cor:SubGraph}.]
Since a word written over $V(\Lambda) \sqcup V(\Lambda)^{-1}$ is $\Lambda$-reduced if and only if it is $\Gamma$-reduced, the desired conclusion follows from Proposition~\ref{prop:NormalForm}. 
\end{proof}

\begin{cor}\label{cor:List}
Let $\Gamma$ be a graph.
\begin{itemize}
	\item[(i)] $A(\Gamma)$ is free if and only if $\Gamma$ has no edge. Otherwise, $A(\Gamma)$ contains $\mathbb{Z}^2$. 
	\item[(ii)] $A(\Gamma)$ is abelian if and only if $\Gamma$ is a complete graph. Otherwise, $A(\Gamma)$ contains $\mathbb{F}_2$.
	\item[(iii)] $A(\Gamma)$ splits non-trivially as a free product if and only if $\Gamma$ is disconnected. If so, $A(\Gamma) = \underset{1 \leq i \leq n}{\ast} A(\Gamma_i)$ where $\Gamma_1, \ldots, \Gamma_n$ are the connected components of $\Gamma$. 
	\item[(iv)] $A(\Gamma)$ splits non-trivially as a direct product if and only if $\Gamma^{\mathrm{opp}}$ is disconnected. If so, $A(\Gamma)= \bigoplus\limits_{1 \leq i \leq n} A(\Gamma_i)$ where $\Gamma_1, \ldots, \Gamma_n$ are the induced subgraphs of $\Gamma$ given by the connected components of $\Gamma^\mathrm{opp}$. 
\end{itemize}
\end{cor}

\noindent
Here, $\Gamma^{\mathrm{opp}}$ refers to the \emph{opposite graph} of $\Gamma$, namely the graph with vertex-set $V(\Gamma)$ whose edges connect two vertices if and only if they are not adjacent in $\Gamma$. Observe that having a disconnected opposite graph amounts to splitting as a \emph{join}, i.e. $\Gamma$ contains two non-empty induced subgraphs $\Phi,\Psi$ such that $V(\Gamma) = V(\Phi) \sqcup V(\Psi)$ and such that every vertex in $\Phi$ is adjacent to every vertex in $\Psi$. We note $\Gamma= \Phi \ast \Psi$. A graph that is not a join is \emph{irreducible}. 

\begin{proof}[Proof of Corollary~\ref{cor:List}.]
If $\Gamma$ has no edge, then $A(\Gamma)$ must be free since it admits a presentation with no relation. Conversely, if $\Gamma$ contains at least one edge, it follows from Corollary~\ref{cor:SubGraph} that $A(\Gamma)$ contains $\mathbb{Z}^2$, so it cannot be free. 

\medskip \noindent
If $\Gamma$ is a complete graph, then $A(\Gamma)$ is clearly free abelian. Conversely, if $\Gamma$ contains two non-adjacent vertices, it follows from Corollary~\ref{cor:SubGraph} that $A(\Gamma)$ contains $\mathbb{F}_2$, so it cannot be abelian.

\medskip \noindent
We refer to Exercices~\ref{ex:OtherItemOne} and~\ref{ex:OtherItemTwo} for the two remaining items.
\end{proof}

\paragraph{CONJUGACY PROBLEM.} Extending the word problem, the conjugacy problem asks, given two words in generators, whether or not the two corresponding elements in the group are conjugate. Similarly to the word problem, it is possible to choose a good representative, which is more or less unique, in each conjugacy class. More precisely, a word $u_1 \cdots u_n$ is \emph{cyclically graphically reduced} if it is graphically reduced and if there do not exist $1 \leq i < j \leq n$ such that $u_i= u_j^{- 1}$ and such that $u_i$ (resp. $u_j$) can be shuffled to the beginning (resp. to the end) of the word. Then:

\begin{prop}\label{prop:ConjProb}
Let $\Gamma$ be a graph. For every $g \in A(\Gamma)$, there exists a cyclically graphically reduced word $w$, unique up to shufflings and cyclic permutations, such that $g$ and $w$ are conjugate in $A(\Gamma)$. 
\end{prop}

\begin{cor}\label{cor:ConjProb}
Let $\Gamma$ be a graph. The conjugacy problem is solvable in $A(\Gamma)$. 
\end{cor}

\begin{proof}
Let $g_1,g_2 \in A(\Gamma)$ be two elements. For $i=1,2$, write $g_i$ as a graphically reduced word $w_i$. One can check algorithmically whether or not $w_i$ is cyclically graphically reduced. If it is not, then we can apply shufflings to $w_i$ to get a new (graphically reduced) word $u \cdot w_i' \cdot u^{-1}$ where $u$ is a generator and where $w_i'$ is a word shorter than $w_i$. Next, we iterate the process with $w_i'$, and so on. After at most $|w_i|$ iterations, we find that $w_i$ can be written as $p_i \cdot w_i'' \cdot p_i^{-1}$ by applying shufflings, where $w_i''$ is now a cyclically graphically reduced word. According to Proposition~\ref{prop:ConjProb}, $g_1$ and $g_2$ are conjugate in $A(\Gamma)$ if and only if $w_1$ and $w_2$ can be obtained from one to the other by applying shufflings and cyclic permutations, which can be checked easily. 
\end{proof}

\noindent
It is worth mentioning that, based on Proposition~\ref{prop:ConjProb}, the conjugacy problem in right-angled Artin groups can be solved very efficiently, namely in linear time \cite{MR2546582}.

\paragraph{COMMUTATION PROBLEM.} Another algorithmic problem that can be solved in right-angled Artin groups is the commutation problem, namely: given two words in generators, determine whether or not the two corresponding elements of the group commute. In order to solve this problem, we will show how to compute the centraliser of an element. In the next statement, the \emph{support} of an element $g$ of right-angled Artin group $A(\Gamma)$ refers to the subgraph of $\Gamma$ given by the generators appearing in a cyclically graphically reduced word representing $g$.

\begin{thm}[\cite{ServatiusCent}]\label{thm:Centralisers}
Let $\Gamma$ be a graph and $g \in A(\Gamma)$ an element. Decompose the support of $g$ as a join $\Lambda_1 \ast \cdots \ast \Lambda_k$ of irreducible subgraphs and write $g$ as $hw_1 \cdots w_k h^{-1}$ where $w_i \in \langle \Lambda_i \rangle$ for every $1 \leq i \leq k$. Then
$$C(g)= h \left( \langle g_1 \rangle \oplus \cdots \langle g_k \rangle \oplus \langle \mathrm{link}(\mathrm{supp}(g)) \rangle \right) h^{-1}$$
where each $g_i$ is a generator of the maximal cyclic subgroup of $\langle \Lambda_i \rangle$ containing $w_i$. 
\end{thm}

\noindent
In this statement, it is implicitly used that every non-trivial element of a right-angled Artin group belongs to a unique maximal infinite cyclic subgroup. This is a consequence of the uniqueness of roots in right-angled Artin groups.

\medskip \noindent
For instance, let $\Gamma$ be the graph containing a $4$-cycle $(a,b,c,d)$ and two vertices $e,f$ adjacent these four vertices. Let us compute the centraliser of $g:=a^{-2}cabdc^{-1}bad$ in $A(\Gamma)$. First, we write
$$g= a^{-1} \cdot a^{-1}cabdc^{-1}bd \cdot a$$
where $a^{-1}cabdc^{-1}bd$ is cyclically graphically reduced. So the support of $g$ is the $4$-cycle $(a,b,c,d)$, which splits as the join $\{a,c\} \ast \{b,d\}$. Writing
$$g= a^{-1} \cdot (a^{-1}cac^{-1}) (bd)^2 \cdot a,$$
we conclude that
$$C(g)= a^{-1} \left( \langle a^{-1}cac^{-1} \rangle \oplus \langle bd \rangle \oplus \langle e,f \rangle \right) a,$$
which is isomorphic to $\mathbb{Z}^2 \oplus \mathbb{F}_2$.

\paragraph{ISOMORPHISM PROBLEM.} Given two graphs, a natural question is to determine whether or not the two corresponding right-angled Artin groups are isomorphic. Of course, if the graphs are isomorphic, then the groups have to be isomorphic. It turns out that the converse is true.

\begin{thm}[\cite{MR891135}]
Let $\Phi,\Psi$ be two finite graphs. The right-angled Artin groups $A(\Phi),A(\Psi)$ are isomorphic if and only if the graphs $\Phi,\Psi$ are isomorphic.
\end{thm}

\paragraph{EMBEDDING PROBLEM.} Many other basic problems are well understood among right-angled Artin groups. One instance of a natural problem that is very far from being understood is the following:
\begin{itemize}
	\item[] \emph{Let $\Phi,\Psi$ be two finite graphs. How to determine whether or not $A(\Phi)$ is isomorphic to a subgroup of $A(\Psi)$?}
\end{itemize}
\noindent
Corollary~\ref{cor:SubGraph} show that, if $\Phi$ can be realised as an induced subgraph of $\Psi$, then $A(\Phi)$ can be realised as a subgroup of $A(\Psi)$. But the converse does not hold, even for basic examples: $\mathbb{F}_3$ (the right-angled Artin group corresponding to three isolated vertices) embeds into $\mathbb{F}_2$ (the right-angled Artin group corresponding to two isolated vertices). 

\medskip \noindent
The goal of the next lectures will be to describe a geometric approach to the Embedding Problem.

\paragraph{A MORE GENERAL FRAMEWORK.} In this mini-course, we focus on right-angled Artin groups, but it turns out that the geometric framework we present makes sense in a more general setting. Given a graph $\Gamma$ and a collection of groups $\mathcal{G}= \{G_u \mid u \in V(\Gamma)\}$ indexed by the vertices of $\Gamma$, the \emph{graph product} $\Gamma \mathcal{G}$ is the group given by the relative presentation
$$\langle G_u, \ u \in V(\Gamma) \mid [G_u,G_v]=1, \ \{u,v\} \in E(\Gamma) \rangle$$
where $[G_u,G_v]=1$ is a shorthand for $[a,b]=1$ for all $a \in G_u$, $b \in G_v$ \cite{GreenGP}. Right-angled Artin groups coincide with graph products of infinite cyclic groups, but one also recovers right-angled Coxeter groups as graph products of cyclic groups of order two. Usually, one says that graph products interpolate between free products (when $\Gamma$ has no edge) and direct sums (when $\Gamma$ is a complete graph). One can also think of graph products of groups as \emph{right-angled rotation groups} \cite{Rotation}.

\addcontentsline{toc}{section}{Exercices}
\section*{Exercices}

\begin{exercice}
Prove that right-angled Artin groups are torsion-free and that their cyclic subgroups are undistorted.
\end{exercice}

\begin{exercice}\label{ex:OtherItemOne}
Prove item $(iii)$ from Corollary~\ref{cor:List}. 

\medskip \noindent
(\emph{Hint: In a free product, free factors are malnormal and non-cyclic free abelian subgroups lie in conjugates of factors.})
\end{exercice}

\begin{exercice}\label{ex:OtherItemTwo}
Prove item $(iv)$ from Corollary~\ref{cor:List}. 

\medskip \noindent
(\emph{Hint: If $\Gamma^\mathrm{opp}$ is connected, use Theorem~\ref{thm:Centralisers} in order to construct an element with cyclic centraliser.})
\end{exercice}

\begin{exercice}\label{ex:WP}
The goal of this exercice is to prove Proposition~\ref{prop:NormalForm}. We begin by recalling some basic knowledge about van Kampen diagrams. 

\begin{definition}
Let $\mathcal{P}= \langle X \mid R \rangle$ be a group presentation, and $D$ a finite $2$-complex embedded into the plane whose edges are oriented and labelled by elements of $X$. An oriented path $\gamma$ in the one-skeleton of $D$ can be written as a concatenation $e_1^{\epsilon_1} \cdots e_n^{\epsilon_n}$, where $\epsilon_1, \ldots, \epsilon_n \in \{+1,-1\}$ and where each $e_i$ is an edge endowed with the orientation coming from $D$. Then the word \emph{labelling} $\gamma$ is $\ell_1^{\epsilon_1} \cdots \ell_n^{\epsilon_n}$ where $\ell_i$ is the label of $e_i$ for every $1 \leq i \leq n$. If, for every $2$-cell $F$ of $D$, the word labelling the boundary of $F$ (an arbitrary basepoint and an arbitrary orientation being fixed) is a cyclic permutation of a relation of $R$ or of the inverse of a relation of $R$, then $D$ is a \emph{diagram over $\mathcal{P}$}.
\end{definition}
\begin{figure}[h!]
\begin{center}
\includegraphics[trim={0 6.5cm 10cm 0},clip,scale=0.34]{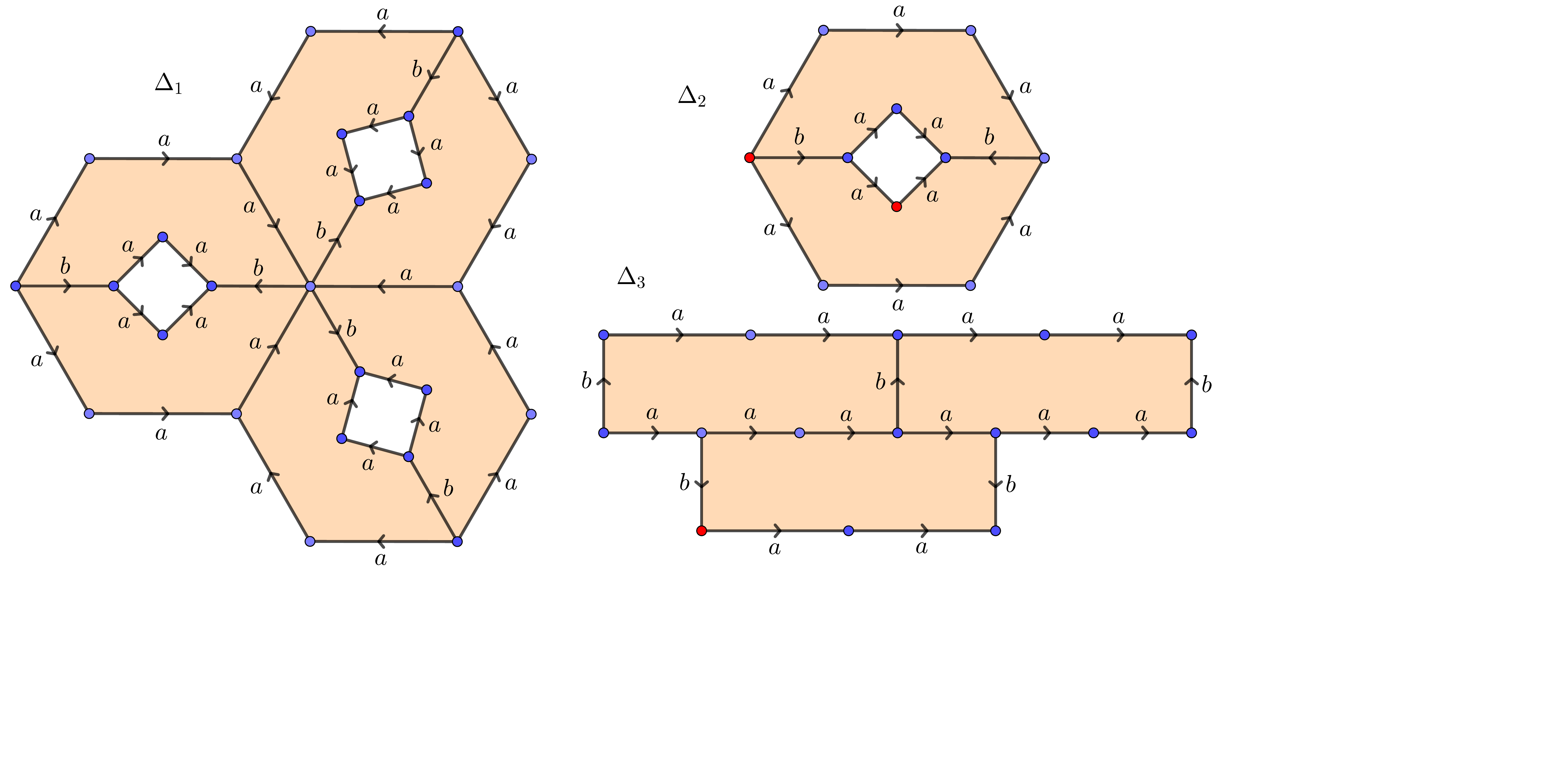}
\caption{A general diagram, a van Kampen diagram, and an annular diagram over the presentation $\langle a,b \mid ba^2b^{-1}=a^3 \rangle$.}
\label{diag}
\end{center}
\end{figure}
\begin{definition}
Let $\mathcal{P}= \langle X \mid R \rangle$ be a group presentation. A \emph{van Kampen diagram over $\mathcal{P}$} is a simply connected diagram over $\mathcal{P}$ with a fixed vertex in its boundary (ie., the intersection between $D$ and the closure of $\mathbb{R}^2 \backslash D$). A \emph{boundary cycle} of $D$ is a cycle $\alpha$ of minimal length which contains all the edges in the boundary of $D$ which does not cross itself, in the sense that, if $e$ and $e'$ are consecutive edges of $\alpha$ with $e$ ending at a vertex $v$, then $e^{-1}$ and $e'$ are adjacent in the cyclically ordered set of all edges of $D$ beginning at $v$. The \emph{label} of the boundary of $D$ is the word labelling the boundary cycle of $D$ which begins at the basepoint of $D$ and which turns around $D$ clockwise. 
\end{definition}

\noindent
The connection between van Kampen diagrams and the word problem is made explicit by the following statement. We refer to \cite[Theorem V.1.1 and Lemma V.1.2]{LS} for a proof.

\begin{prop}
Let $\mathcal{P}= \langle X \mid R \rangle$ be a presentation of a group $G$ and $w \in X^{\pm}$ a non-empty word. There exists a van Kampen diagram over $\mathcal{P}$ whose boundary is labelled by $w$ if and only if $w=1$ in $G$.
\end{prop}

\noindent
Now, we fix a graph $\Gamma$ and the presentation $\mathcal{P}$ of $A(\Gamma)$ given by Definition~\ref{def:RAAG}. Observe that diagrams over $\mathcal{P}$ are square complexes. A \emph{dual curve} in such a diagram is minimal subset obtained by concatenating straight lines that connect midpoints of opposite edges in squares. 
\begin{enumerate}
	\item In a diagram over $\mathcal{P}$, two edges crossed by the same dual curve are labelled by the same vertex of $\Gamma$. Thus, dual curves are naturally labelled by vertices of $\Gamma$.
	\item If two dual curves cross, then they are labelled by adjacent vertices of $\Gamma$. As a consequence, a dual curve cannot self-intersect.
	\item In a diagram over $\mathcal{P}$, a path labelled by a graphically reduced word cannot cross a dual curve twice. 
	\item Deduce that graphically reduced words coincide with words of minimal length.
	\item Prove that two graphically reduced words representing the same element of $A(\Gamma)$ differ only by finitely many shufflings.
\end{enumerate}
\end{exercice}

\newpage

\section{Lecture 2: A crash course on (quasi-)median graphs}

\noindent
In this lecture, we think of a graph as a set of vertices endowed with an adjacency relation. As a consequence, a point in a graph always refers to a vertex. In other words, we distinguish an abstract graph from its geometric realisation, so points on edges do not exist for us. Equivalently, a graph can be thought of as a metric space whose metric takes values in $\mathbb{N} \cup \{+\infty\}$ and such that any two finite points at finite distance can be connected by a sequence of points that are successively at distance one.

\medskip \noindent
Given a graph $X$, a subgraph $Y$ is a graph whose vertex- and edge-sets are included in the vertex- and edge-sets of $X$. We distinguish three types of subgraphs:
\begin{itemize}
	\item $Y$ is \emph{induced} if any two vertices in $Y$ that are adjacent in $X$ are also adjacent in~$Y$;
	\item $Y$ is \emph{convex} if any geodesic between any two vertices in $Y$ lies in $Y$;
	\item $Y$ is \emph{gated} if, for every $x \in X$, there exists some $z \in Y$ (referred to as the \emph{gate} or \emph{projection}) that belongs to at least one geodesic from $x$ to $y$ for every $y \in Y$.
\end{itemize}
\begin{minipage}{0.49\linewidth}
\begin{center}
\includegraphics[width=0.6\linewidth]{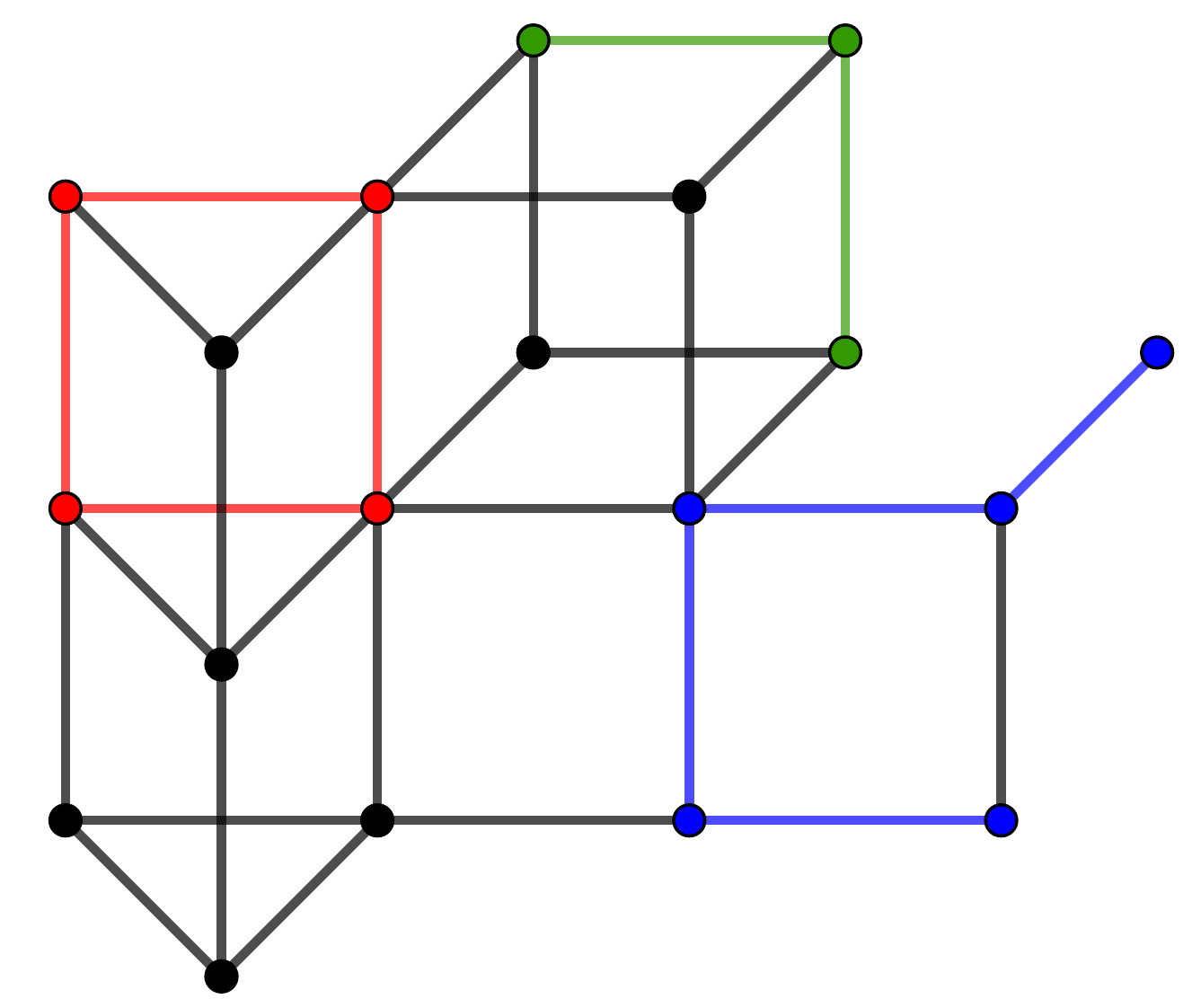}
\end{center}
\end{minipage}
\begin{minipage}{0.4\linewidth}
In blue, a subgraph that is not induced; in green, a subgraph that is induced but not convex; in red, a subgraph that is convex but not gated.
\end{minipage}

\medskip \noindent
Observe that a gated subgraph is automatically convex, and that the projection of a vertex $x$ coincides with the unique vertex in $Y$ minimising the distance from $x$. Consequently, gatedness can be thought of as a strong convexity condition.

\begin{definition}
A connected graph $X$ is \emph{median} if and only if, for any three vertices $x_1,x_2,x_3 \in X$, there exists a unique vertex $m \in X$ (referred to as the \emph{median point}) that satisfies
$$d(x_i,x_j) = d(x_i,m)+d(m,x_j), \ \forall i \neq j.$$
\end{definition}

\noindent
The typical example to keep in mind is that of a tree. But many other graphs, quite different from trees, are median.
\begin{center}
\includegraphics[width=\linewidth]{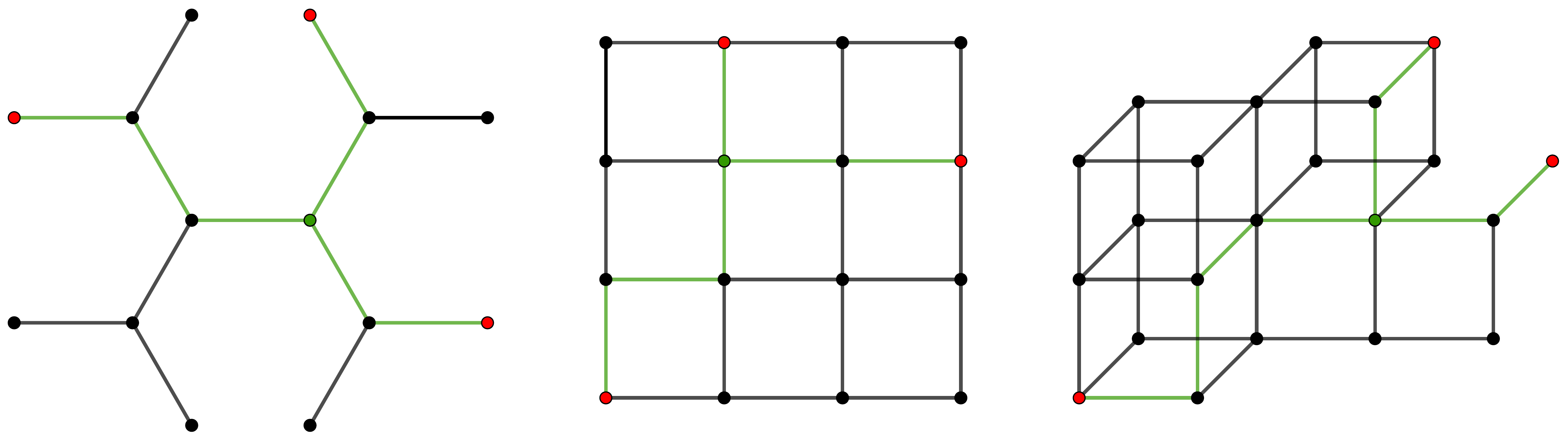}
\end{center}
It turns out that median graphs coincide with one-skeleta of \emph{CAT(0) cube complexes}. Roughly speaking, median graphs can be thought of as one-skeleta of nonpositively curved cellular complexes made of cubes of arbitrary dimension.

\medskip \noindent
Typical counterexamples of median graphs are given by a $3$-cycle (whose three vertices do not admit a median point) and by bipartite complete graph $K_{3,2}$ (whose three pairwise non-adjacent vertices admit two median points). Interestingly, it is possible to define a larger family of graphs that are allowed to contain $3$-cycles and that have essentially the same geometry as median graphs.

\begin{definition}
Let $X$ be a connected graph. Given three vertices $x_1,x_2,x_3 \in X$, a \emph{median triangle} is the data of three vertices $y_1,y_2,y_3 \in X$ such that
$$d(x_i,x_j)=d(x_i,y_i)+d(y_i,y_j)+d(y_j,x_j), \ \forall i \neq j.$$
Its \emph{size} if $d(y_1,y_2)+d(y_2,y_3)+d(y_3,y_1)$. The graph $X$ is \emph{quasi-median} if any three vertices admit a unique median triangle of minimal size and if the gated hull of any such median triangle is a product of complete graphs.
\end{definition}
\begin{center}
\includegraphics[width=\linewidth]{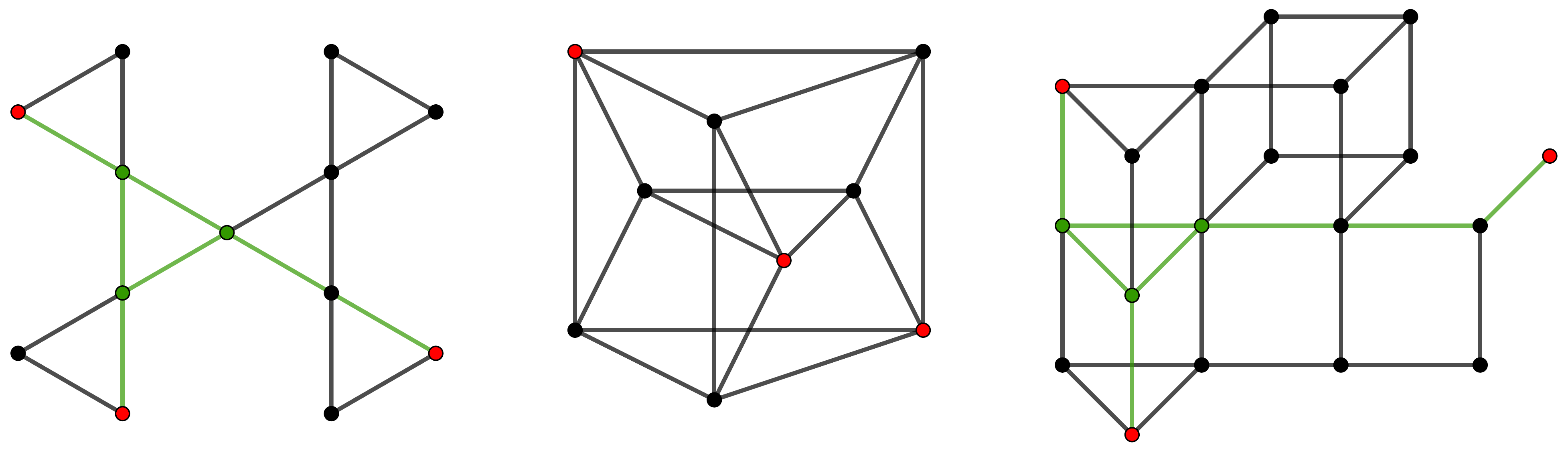}
\end{center}
\noindent
Observe that, in an arbitrary graph $X$, three vertices $x_1,x_2,x_3 \in X$ always admit a median triangle, namely $x_1,x_2,x_3$. What is interesting is to find a median triangle of small size. Quasi-median graphs are defined by requiring the smallest median triangle to be unique and to be as small as possible in a graph containing $3$-cycles. (Observe that a product $P$ of complete graphs is a quasi-median graph, and that it contains three vertices $a,b,c$ whose median triangle is $a,b,c$ and whose gated hull is the whole graph~$P$.)

\medskip \noindent
The graph $K_4^-$, namely the graph obtained from the complete graph $K_4$ by removing an edge (which amounts to saying that $K_4^-$ is obtained from by gluing two $3$-cycles along an edge), is an example of a graph that is not quasi-median. Indeed, if $a,b,c \in K_4^-$ are the three vertices of a $3$-cycle, then $a,b,c$ is the corresponding median triangle and its gated hull is the whole $K_4^-$. But $K_4^-$ is not a product of complete graphs, so it cannot be quasi-median. Cycles of length $\geq 5$ are other examples of graphs that are not quasi-median. 

\medskip \noindent
The idea that quasi-median graphs are ``median graphs with $3$-cycles'' is motivated by Exercice~\ref{ex:MvsQM}. Also, in the same way that median graphs are one-skeleta of CAT(0) cube complexes, quasi-median graphs are one-skeleta of CAT(0) prism complexes (a prism referring to a product of simplices).

\medskip \noindent
Median and quasi-median graphs are both share a common property: a parallelism relation on edges allows us to define classes of edges, namely \emph{hyperplanes}, that separate our graph and that encodes, in some sense, its geometry. 

\begin{definition}
In a quasi-median graph, a \emph{hyperplane} is an equivalence class of edges with respect to the transitive closure of the relation that identifies any two edges that belong to the same $3$-cycle or that are opposite sides of a $4$-cycle. Two hyperplanes are \emph{transverse} if they contain two distinct pairs of opposite sides in some $4$-cycle.
\end{definition}
\begin{center}
\includegraphics[trim={0cm 16cm 10cm 0cm},clip,width=0.5\linewidth]{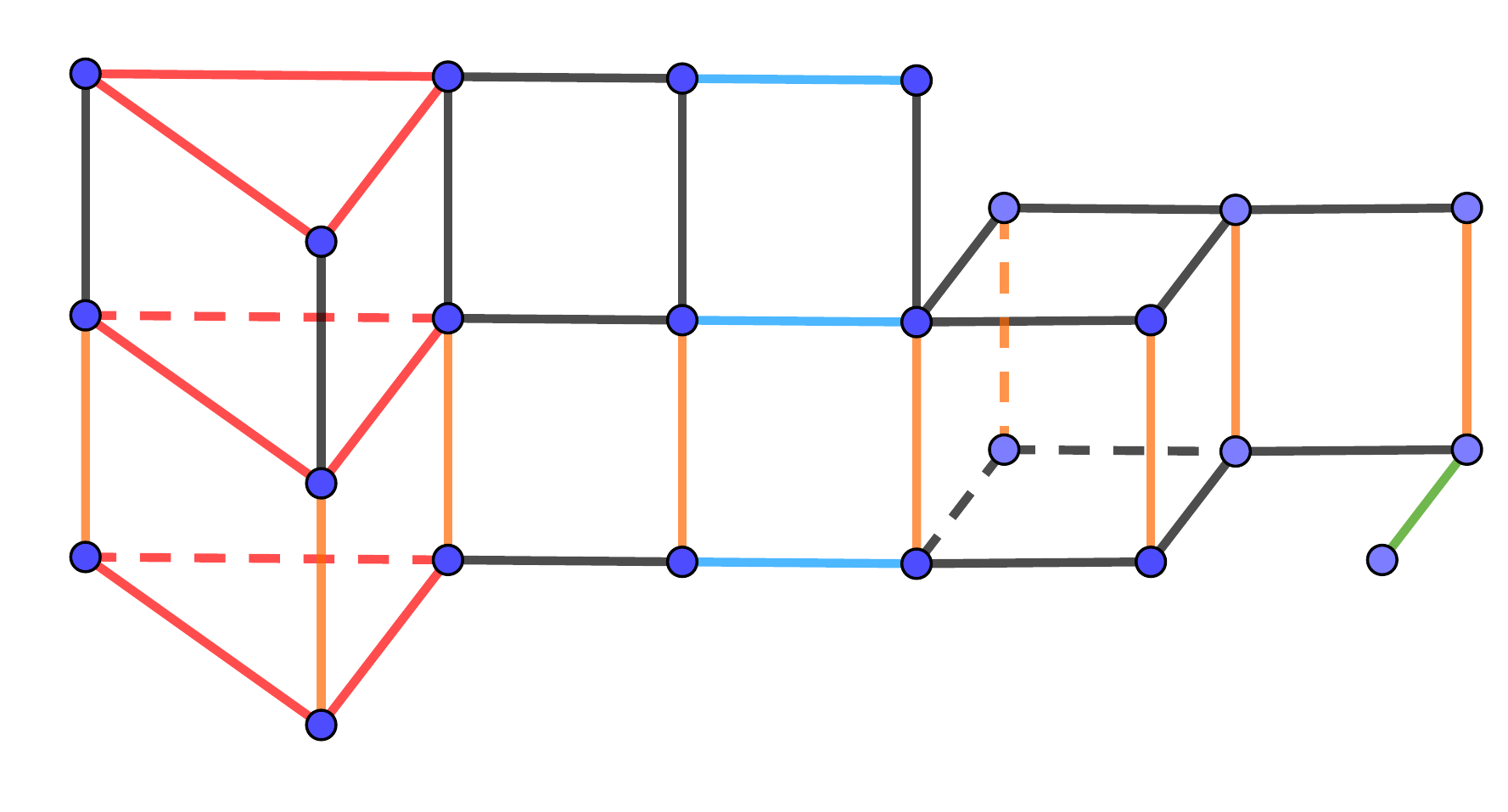}
\end{center}
\noindent
The idea that the combinatorics of the hyperplanes encodes the geometry of a quasi-median graphs is motivated by the following statement:

\begin{thm}[\cite{QM}]
Let $X$ be a quasi-median graph. 
\begin{itemize}
	\item For every hyperplane $J$, the graph $X \backslash \backslash J$ obtained from $X$ by removing the edges in $J$ has at least two connected components, referred to as \emph{sectors}.
	\item Sectors are gated.
	\item A path in $X$ is a geodesic if and only it crosses each hyperplane at most once. As a consequence, the distance between any two vertices coincides with the number of hyperplanes separating them.
\end{itemize}
\end{thm}

\medskip \noindent
It turns out that examples of median and quasi-median graphs naturally include Cayley graphs of right-angled Artin groups. Such median graphs are extremely useful in the study of right-angled Artin groups for many problems, but quasi-median graphs will play the central role for the Embedding Problem. 

\begin{prop}
Let $\Gamma$ be a graph. Then $\mathrm{M}(\Gamma):= \mathrm{Cayl}(A(\Gamma),V(\Gamma))$ is a median graph and $\mathrm{QM}(\Gamma):= \mathrm{Cayl}(A(\Gamma), \{u^k \mid u \in V(\Gamma), k \in \mathbb{Z}_{\neq 0} \} )$ a quasi-median graph.
\end{prop}

\begin{proof}[Idea of proof.]
One needs to show that, given three elements of $A(\Gamma)$, there exist a unique median point in $\mathrm{M}(\Gamma)$ and a unique minimal median triangle in $\mathrm{QM}(\Gamma)$. Up to translating by an element of $A(\Gamma)$, we can assume without loss of generality that $1$ is one of these elements. As an illustration, assume that $\Gamma$ is a $5$-cycle $(a,b,c,d,e)$ and that our three elements are $1$, $g:= ab^2ce$, $h:=ab^{-1}c^{-2}a$. 

\medskip \noindent
The median point $m$ of $1,g,h$ in $\mathrm{M}(\Gamma)$, if it exists, has to belong to geodesics from $1$ to $g,h$.  In other words, in terms of graphically reduced words, $m$ must be a common prefix of $g$ and $h$. So we define $m$ as the longest common prefix of $g$ and $h$, and we claim that it defines a median point. In our example, $m=a$. In order to justify that $a$ is indeed a median point, we need to show that $a$ lies on a geodesic between $g$ and $h$, which amounts to saying that $ (b^2ce)^{-1} \cdot b^{-1}c^{-2}a$ has minimal length (with respect to the generating set $V(\Gamma)$). One can check the word is graphically reduced.
\begin{center}
\includegraphics[width=\linewidth]{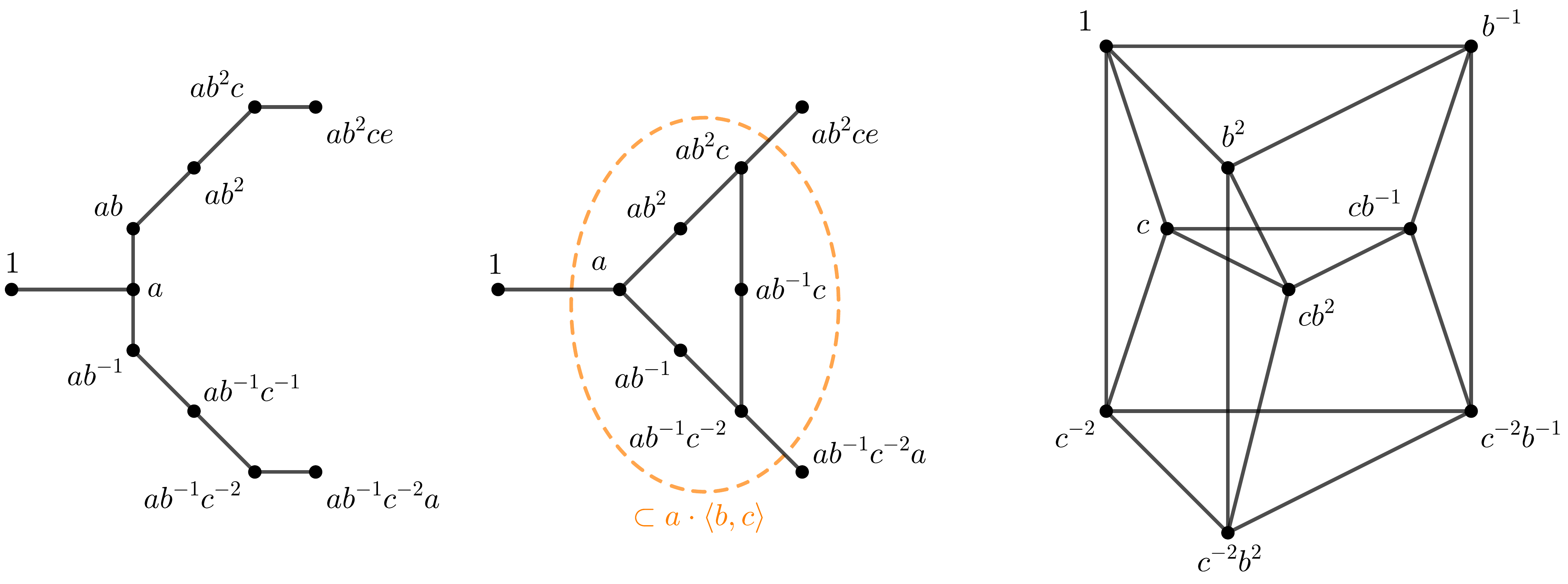}
\end{center}
\noindent
However, the word does not have minimal length with respect to $\{u^k \mid u \in V(\Gamma), k \in \mathbb{Z}_{\neq 0} \}$. Geometrically, observe that the vertices $ab^2$ and $ab^{-1}$ are adjacent in $\mathrm{QM}(\Gamma)$. The longest common prefix $a$ gives one of the three vertices of our median triangle. The two other vertices are obtained as follows. In addition to the longest common prefix, we add the letters that arise as first letters in $m^{-1}g$ and $m^{-1}h$ with different powers. In our examples, we find $ab^2c$ and $ab^{-1}c^{-2}$. Observe that these additional letters are given by vertices of $\Gamma$ that are pairwise adjacent. As a consequence, the vertices $a$, $ab^2c$, $ab^{-1}c^{-2}$ lie in $a \cdot \langle b,c \rangle = a \cdot \langle b \rangle \oplus \langle c \rangle$. One can check that $a \langle b,c \rangle$ coincides with the gated hull of our three vertices, and it is clearly a product of complete graphs. Moreover, $ab^2c$ and $ab^{-1}c^{-2}$ lies on a geodesic from $ab^2ce$ to $ab^{-1}c^{-2}a$ (in that order), which amounts to saying that the word $(ce)^{-1} \cdot a$ has minimal length with respect to our generating~set. 
\end{proof}

\medskip \noindent
Let us conclude by recording some information on the structure of the quasi-median graph $\mathrm{QM}(\Gamma)$. Recall that a \emph{clique} is a maximal complete subgraph; a \emph{prism} is a subgraph that splits as a product of cliques; and the \emph{star} of a vertex is the subgraph given by the vertex under consideration together with all its neighbours.

\begin{lemma}\label{lem:QMGamma}
Let $\Gamma$ be a graph. The following assertions hold.
\begin{itemize}
	\item The cliques of $\mathrm{QM}(\Gamma)$ coincides with the cosets $g \langle u \rangle$ where $g \in A(\Gamma)$ and $u \in V(\Gamma)$.
	\item The prisms of $\mathrm{QM}(\Gamma)$ coincides with the cosets $g \langle \Lambda \rangle$ where $g \in A(\Gamma)$ and where $\Lambda \leq \Gamma$ is complete.
	\item For every $u \in V(\Gamma)$, let $J_u$ denote the hyperplane containing the clique $\langle u \rangle$. The hyperplanes of $\mathrm{QM}(\Gamma)$ are the $gJ_u$ where $g \in A(\Gamma)$ and $u \in V(\Gamma)$, with $g \langle \mathrm{star}(u) \rangle g^{-1}$ as stabiliser. 
\end{itemize}
\end{lemma}

\noindent
The lemma is essentially a consequence of the characterisation of $3$- and $4$-cycles in $\mathrm{QM}(\Gamma)$. If $a,b,c \in \mathrm{QM}(\Gamma)$ is a $3$-cycle, then there exist generators $u,v$ and integers $r,s \in \mathbb{Z}_{\neq 0}$ such that $b=au^r$ and $c=av^s$, just because $b$ and $c$ are both adjacent to $a$. But $b$ and $c$ are also adjacent, so $u^{-r}v^s$ has to be a power of a generator. As a consequence of the normal in right-angled Artin groups, we must have $u=v$. Thus, $3$-cycles in $\mathrm{QM}(\Gamma)$ lie in cosets $g \langle u \rangle \rangle$ where $g \in A(\Gamma)$ and $u \in V(\Gamma)$. Next, if $a,b,c,d$ is an induced $4$-cycle, then there must exist generators $u,v \in V(\Gamma)$ and non-zero integers $r,s \in \mathbb{Z}$ such that $b=au^r$ and $d=av^s$ just because $b$ and $d$ are both adjacent to $a$. But they are also both adjacent to $c$, so there must exist two other generators $f,h \in V(\Gamma)$ and two non-zero integers $p,q \in \mathbb{Z}$ such that $au^rf^p = c = av^sh^q$. From $u^rf^p = v^sh^q$, we deduce again from the normal form in right-angled Artin groups that $u=h$, $f=v$, $p=s$, $q=r$, and that $u,v$ are adjacent in $\Gamma$. Thus, $4$-cycles in $\mathrm{QM}(\Gamma)$ lie in cosets $g \langle u, v \rangle$ where $g \in A(\Gamma)$ and where $u,v \in V(\Gamma)$ are adjacent.

\medskip \noindent
We leave the details as an exercice.

\addcontentsline{toc}{section}{Exercices}
\section*{Exercices}

\begin{exercice}\label{ex:MvsQM}
Show that a graph is median if and only if it is quasi-median and has no $3$-cycle. Show that, in a median graph, a subgraph is gated if and only if it is convex.
\end{exercice}

\begin{exercice}
Let $X$ be a quasi-median graph.
\begin{enumerate}
	\item Let $o,x,y \in X$ be three vertices such that $x,y$ are adjacent and $d(o,x)=d(o,y)$. Show that there exists a vertex $z \in X$ that is adjacent to both $x,y$ and that satisfies $d(o,z)=d(o,x)-1$.
	\item Let $o,x,y,z \in X$ be four vertices such that $y$ is adjacent to both $x,z$, such that $d(x,z)=2$, and such that $d(o,x)=d(o,z)=d(o,y)-1$. Show that there exists a vertex $w \in X$ that is adjacent to both $x,y$ that satisfies $d(o,w)=d(o,y)-2$. 
	\item Show that $X$ does not contain an induced subgraph isomorphic to the bipartite complete graph $K_{3,2}$ nor to $K_4^-$ (i.e. a complete graph $K_4$ with one edge removed).
\end{enumerate}
A graph satisfying the first two conditions is said \emph{weakly modular}. Usually, quasi-median graphs are defined as weakly modular graphs without induced $K_{3,2}$, $K_4^-$.
\end{exercice}

\begin{exercice}\label{ex:ConvexityInQM}
Let $X$ be a quasi-median graph. Given a path $(x_0, \ldots, x_n)$ in $X$, one
\begin{itemize}
	\item \emph{removes a backtrack} if $x_i=x_{i+2}$ for some $i$ and if we remove $x_{i+1},x_{i+2}$;
	\item \emph{shortens a $3$-cycle} if $x_i,x_{i+2}$ are adjacent for some $i$ and if we remove $x_{i+1}$;
	\item \emph{flips a $4$-cycle} if $[x_{i+1},x_i],[x_{i+1},x_{i+2}]$ span a $4$-cycle and if we replace $x_{i+1}$ with the vertex opposite to it in this $4$-cycle.
\end{itemize}
The goal of this exercice is to show that a subgraph $Y$ in $X$ is convex if and only if it is connected and \emph{locally convex} (i.e. every $4$-cycle having two consecutive edges in $Y$ must lie entirely in $Y$). 
\begin{enumerate}
	\item Show that any two geodesics with the same endpoints differ only by finitely many flips of $4$-cycles.
	\item Show that an arbitrary path can be turned into a geodesic by shortening $3$-cycles, by flipping $4$-cycles, and by removing backtrack.
	\item Show that a subgraph in $X$ is convex if and only if it is connected and locally convex. 
\end{enumerate}
\end{exercice}

\begin{exercice}
Let $X$ be a quasi-median graph. 
\begin{enumerate}
	\item Given a vertex $x \in X$ and a gated subgraph $Y \subset X$, show that every hyperplane separating $x$ from its projection on $Y$ separates $x$ from $Y$.
	\item Given two vertices $x,y \in X$ and a gated subgraph $Y \subset X$, show that the hyperplanes separating the projections of $x,y$ on $Y$ coincide with the hyperplanes separating $x,y$ that cross $Y$. Deduce that the projection on $Y$ is $1$-Lipschitz.
	\item Fix a collection of cliques $\mathcal{C}$ such that every hyperplane of $X$ contains exactly one clique in $\mathcal{C}$. Prove that $$\left\{ \begin{array}{ccc} X & \to & \prod\limits_{C \in \mathcal{C}} C \\ x & \mapsto & (\mathrm{proj}_C(x))_{C \in \mathcal{C}} \end{array} \right.$$ defines an isometric embedding. Show that the image may not be convex. 
\end{enumerate}
\end{exercice}

\begin{exercice}
Show Cartesian products and pointed sums of (quasi-)median graphs are (quasi-)median.
\end{exercice}

\begin{exercice}
Fix a graph $\Gamma$. The goal of this exercice is to prove Lemma~\ref{lem:QMGamma}. 
\begin{enumerate}
	\item Show that the edges of a $3$-cycle in $\mathrm{QM}(\Gamma)$ are all labelled by the same vertex of~$\Gamma$. Deduce the first item of Lemma~\ref{lem:QMGamma}.
	\item Show that, in a $4$-cycle, two opposite edges are labelled by the same vertex of $\Gamma$. Deduce the second item of Lemma~\ref{lem:QMGamma}.
	\item Given a vertex $u \in V(\Gamma)$, show that an edge in $\mathrm{QM}(\Gamma)$ belongs to $J_u$ if and only if it is of the form $\{g,gu^k\}$ for some $k \in \mathbb{Z}_{\neq 0}$ and some $g \in \langle \mathrm{link}(u) \rangle$. Deduce the third item of Lemma~\ref{lem:QMGamma}.
\end{enumerate}
\end{exercice}

\begin{exercice}
Let $X$ be an arbitrary median graph. Let $\Gamma$ denote its \emph{crossing graph}, i.e. the graph whose vertices are the hyperplanes in $X$ and whose edges connect two hyperplanes whenever they are transverse. The goal of this exercice is to show that $X$ embeds into $\mathrm{M}(\Gamma)$ as a convex subgraph.
\begin{enumerate}
	\item Fix an orientation for each hyperplane in $X$. Given an oriented edge $e$ and the hyperplane $J$ containing it, define the label $\ell(e)$ of $e$ as $J$ if the orientations of $e$ and $J$ agree, and $J^{-1}$ otherwise. Thus, an oriented path in $X$ is naturally labelled by a word of hyperplanes (and their inverses). Show that the word labelling a geodesic is graphically reduced.
	\item Show that the words labelling two paths with the same endpoints represent the same element of $A(\Gamma)$. (\emph{Hint: use Exercice~\ref{ex:ConvexityInQM}.})
	\item Fix a basepoint $x_0 \in X$. Show that the map $$\left\{ \begin{array}{ccc} X & \mapsto & \mathrm{M}(\Gamma) \\ x & \mapsto & \ell(\text{path from $x_0$ to $x$}) \end{array} \right.$$ induces an isometric embedding.
	\item Verify that the image of this embedding is convex. (\emph{Hint: use Exercice~\ref{ex:ConvexityInQM}.})
	\item Adapt the construction in order to show that a quasi-median graph $X$ can be realised as a convex subgraph in $\mathrm{QM}(\Gamma)$ where $\Gamma$ denotes the crossing graph of $X$. 
\end{enumerate}
\end{exercice}

\newpage

\section{Lecture 3: Creating RAAGs as subgroups}

\noindent
A well-known tool that is quite useful to construct free subgroups is the ping-pong lemma. It turns out that such a lemma naturally extends to right-angled Artin groups (\cite{CF}, see also \cite{MR2987617, QM}):

\begin{prop}\label{prop:PingPong}
Let $G$ be a group acting on some set $S$. Fix a generating set $R \subset G$ and assume that each $r \in R$ has an associated subset $S_r \subset S$ such that:
\begin{itemize}
	\item if $a,b \in R$ commute, then $a^k \cdot S_b \subset S_b$ for every $k \in \mathbb{Z}_{\neq 0}$;
	\item if $a,b \in R$ do not commute, then $a^k \cdot S_b \subset S_a$ for every $k \in \mathbb{Z}_{\neq 0}$;
	\item there exists a point $s_0 \in S \backslash \bigcup_{r \in R} S_r$ such that $a^k \cdot s_0 \in S_a$ for every $k \in \mathbb{Z}_{\neq 0}$.
\end{itemize}
Then $G$ is a right-angled Artin group. More precisely, if $C_R$ denotes the commutation graph of $R$ (i.e. $C_R$ is the graph whose vertex-set is $R$ and whose edges connect two generators if they commute), then the identity $R \to R$ extends to an isomorphism $A(C_R) \to G$. 
\end{prop}

\begin{proof}
The identity map $R \to R$ extends to a surjective morphism $\varphi : A(C_R)\to G$, but we have to verify that it is injective. Let $g \in A(C_R)$ be a non-trivial element. We can represent it as a graphically reduced word $w$ over $R \cup R^{-1}$ (thought of as elements of $A(C_R)$), and its image under $\varphi$ is represented by the same word over $R \cup R^{-1}$ (thought of as elements of $G$). We want to prove by induction over the length of $w$ that $w \cdot s_0$ belongs to $S_r$ for some $r \in \mathrm{head}(w)$. Here, $\mathrm{head}(w)$ refers to the set of letters of $w$ that can be shuffled to the beginning of the word. 

\medskip \noindent
Write $w$ as $r^k w'$ for some $r \in \mathrm{head}(w)$ where $k \in \mathbb{Z}_{\neq 0}$ is such that $r \notin \mathrm{head}(w')$. If $w'$ is empty, then $w \cdot x_0 \in S_r$ by assumption. Otherwise, there exists $t \in \mathrm{head}(w')$ such that $w' \cdot s_0 \in S_t$. We distinguish two cases: either $r,t$ are adjacent in $\Gamma$, and it follows that $t$ also belongs to $\mathrm{head}(w)$ and that $w \cdot s_0 \in r^k \cdot S_t \subset S_t$; or $r,t$ are not adjacent in $C_R$, and it follows that $w \cdot s_0 \in r^k \cdot S_t \subset S_r$. This concludes the proof of our claim.

\medskip \noindent
It follows that $w \cdot s_0 \neq s_0$ since $s_0$ does not belong to $\bigcup_{r \in R} S_r$, proving that $\varphi(g) \neq 1$. 
\end{proof}

\noindent
As an application, it is possible to construct right-angled Artin groups in groups acting on quasi-median graphs. Before stating our main result in this direction, we need to introduce some terminology.

\begin{definition}
Let $G$ be a group acting on a quasi-median graph $X$. The \emph{rotative-stabiliser} $\mathrm{stab}_\circlearrowright(J)$ of a hyperplane $J$ is the intersection of the stabilisers of all the cliques in $J$. 
\end{definition}

\noindent
The picture to keep in mind is the following. Consider a group $A \oplus \mathbb{Z}$ acting on its Cayley graph $X:=\mathrm{Cayl}(A \oplus \mathbb{Z}, A \cup \{t\})$ where $t$ is a generator of the $\mathbb{Z}$-factor. Geometrically, $X$ is the product of a complete graph with $A$ as its vertex-set and a bi-infinite line. Let $J$ be the hyperplane containing the clique given by $A$. 
\begin{center}
\includegraphics[width=\linewidth]{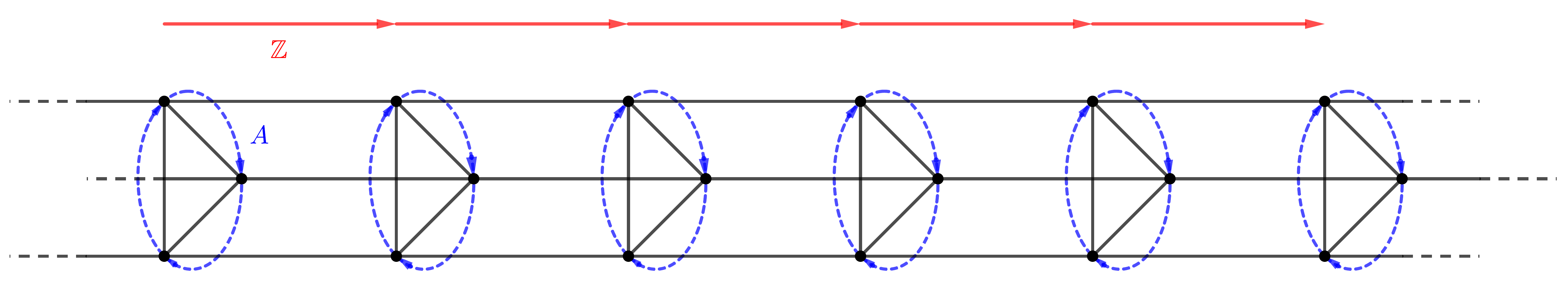}
\end{center}
Then $A \oplus \mathbb{Z}$ stabilises $J$, but the two factors have very different dynamics: the $\mathbb{Z}$-factor stabilises each factor delimited by $J$ and translates the cliques of $J$, while the $A$-factor stabilises each clique in $J$ and permutes the sectors delimited by $J$. The rotative-stabiliser of $J$ coincides with the $A$-factor.

\begin{prop}[\cite{QM}]\label{prop:SubRAAG}
Let $G$ be a group acting on a quasi-median graph $X$ and $\mathcal{J}$ a finite collection of hyperplanes. Assume that:
\begin{itemize}
	\item vertex-stabilisers are trivial;
	\item for every $J \in \mathcal{J}$, there exists a non-trivial element $r_J\in \mathrm{stab}_\circlearrowright(J)$ that permutes freely the sector delimited by $J$. 
\end{itemize}
There exists some $N \geq 1$ such that, for every $n \geq N$, the subgroup $H:=\langle r_J^n, \ J \in \mathcal{J} \rangle$ is a right-angled Artin group. More precisely, if $\mathscr{C} \mathcal{J}$ is the crossing graph of $\mathcal{J}$, then the map $J \mapsto r_J^n$ induces an isomorphism $A(\mathscr{C} \mathcal{J}) \to H$. 
\end{prop}

\noindent
Here, the \emph{crossing graph} $\mathscr{C}\mathcal{J}$ refers to the graph whose vertex-set is $\mathcal{J}$ and whose edges connect two hyperplanes whenever they are transverse. 

\begin{proof}[Proof of Proposition~\ref{prop:SubRAAG}.]
Fix a basepoint $x_0 \in X$. Given a hyperplane $J \in \mathcal{J}$, let $S_J$ denote the union of the sectors delimited by $J$ that do not contain any hyperplane in $\mathcal{J}$ nor $x_0$. Because there are only finitely many sectors delimited by $J$ that do not belong to $S_J$ (as $\mathcal{J}$ is finite), there exists some $n_J \geq 1$ such that every power of $r_J^{n_J}$ sends a sector not in $S_J$ to a sector in $S_J$. Set $N:= \max \{n_J \mid J \in \mathcal{J} \}$ and fix some $n \geq N$. 

\medskip \noindent
Now, we want to apply Proposition~\ref{prop:PingPong} to $x_0$, $\{r^n_J \mid J \in \mathcal{J} \}$, and $\{S_J \mid J \in \mathcal{J}\}$. Observe that:
\begin{itemize}
	\item if $A,B \in \mathcal{J}$ are transverse, then $r_A^{nk} \cdot S_B=S_B$ for every $k \in \mathbb{Z}_{\neq 0}$;
	\item if $A,B \in \mathcal{J}$ are not transverse, then $r_A^{nk} \cdot S_B \subset S_A$ for every $k \in \mathbb{Z}_{\neq 0}$;
	\item $x_0$ does not belong to $\bigcup_{J \in \mathcal{J}} S_J$ and $r_J^{nk} \cdot x_0 \in S_J$ for all $J \in \mathcal{J}$ and $k \in \mathbb{Z}_{\neq 0}$. 
\end{itemize}
Thus, the desired conclusion follows from the next assertion:

\begin{fact}
For any two distinct $A,B \in \mathcal{J}$, $\langle r_A,r_B \rangle = \langle r_A \rangle \oplus \langle r_B \rangle$ if $A,B$ are transverse and $\langle r_A,r_B \rangle= \langle r_A \rangle \ast \langle r_B \rangle$ otherwise.
\end{fact}

\noindent
If $A,B$ are not transverse, the decomposition $\langle r_A,r_B \rangle = \langle r_A \rangle \ast \langle r_B \rangle$ follows from what we have already said by a classical ping-pong argument. Assume that $A,B$ are transverse. So there exists a prism $C \times D$ with $C \subset A$ and $D \subset B$. Because $r_A$ (resp. $r_B$) acts freely on $C$ (resp. $D$) and acts trivially on $D$ (resp. $C$), and that vertex-stabilisers are trivial, we conclude that $\langle r_A,r_B \rangle= \langle r_A \rangle \ast \langle r_B \rangle$ as desired.
\end{proof}

\noindent
Proposition~\ref{prop:SubRAAG} holds for arbitrary groups acting on quasi-median graphs, so let us restrict its statement to right-angled Artin groups.

\begin{cor}\label{cor:KK}\footnote{As shown in \cite{MR3072113}, the converse of the corollary may not hold in full generality.}
Let $\Gamma$ be a graph. Let $\mathscr{C}(\Gamma)$ denote the crossing graph of all the hyperplanes in $\mathrm{QM}(\Gamma)$. For every finite induced subgraph $\Lambda \leq \mathscr{C}(\Gamma)$, $A(\Gamma)$ contains a subgroup isomorphic to $A(\Lambda)$.
\end{cor}

\begin{proof}
The assertion is a straightforward application of Proposition~\ref{prop:SubRAAG} to the action of $A(\Gamma)$ on $\mathrm{QM}(\Gamma)$. 
\end{proof}

\noindent
One can easily describe the crossing graph $\mathscr{C}(\Gamma)$ as follows\footnote{It is worth noticing that this description coincides exactly with the definition of the \emph{extension graph} given in \cite{MR3039768}. Thus, Corollary~\ref{cor:KK} coincides with \cite[Theorem~1.3]{MR3039768}.}:
\begin{itemize}
	\item the vertices of $\mathscr{C}(\Gamma)$ are the $A(\Gamma)$-conjugates of the generators in $V(\Gamma)$;
	\item the edges of $\mathscr{C}(\Gamma)$ connect two conjugates if they commute.
\end{itemize}
This description follows from the fact that the hyperplanes in $\mathrm{QM}(\Gamma)$ are the $gJ_u$, where $g \in A(\Gamma)$ and where $J_u$ denotes the hyperplane containing the clique $\langle u \rangle \leq \mathrm{QM}(\Gamma)$; that the rotative-stabiliser of $gJ_u$ is $g \langle u \rangle g^{-1}$; and that two hyperplanes in $\mathrm{QM}(\Gamma)$ are transverse if and only if their rotative-stabilisers commute, otherwise they span a free group. We leave the details as an exercice. 

\medskip \noindent
As an illustration, let us construct non-trivial embeddings between right-angled Artin groups. See also Exercice~\ref{ex:EmbeddingTree}. 

\begin{ex}
Let $\Gamma$ be a $5$-cycle $(a,b,c,d,e)$. By rotating the hyperplanes $J_c,J_d,J_e,J_a$ around $J_b$, we find a $6$-cycle of hyperplanes $J_c,J_d,J_e,J_a,bJ_e,bJ_d$.
\begin{center}
\includegraphics[width=0.5\linewidth]{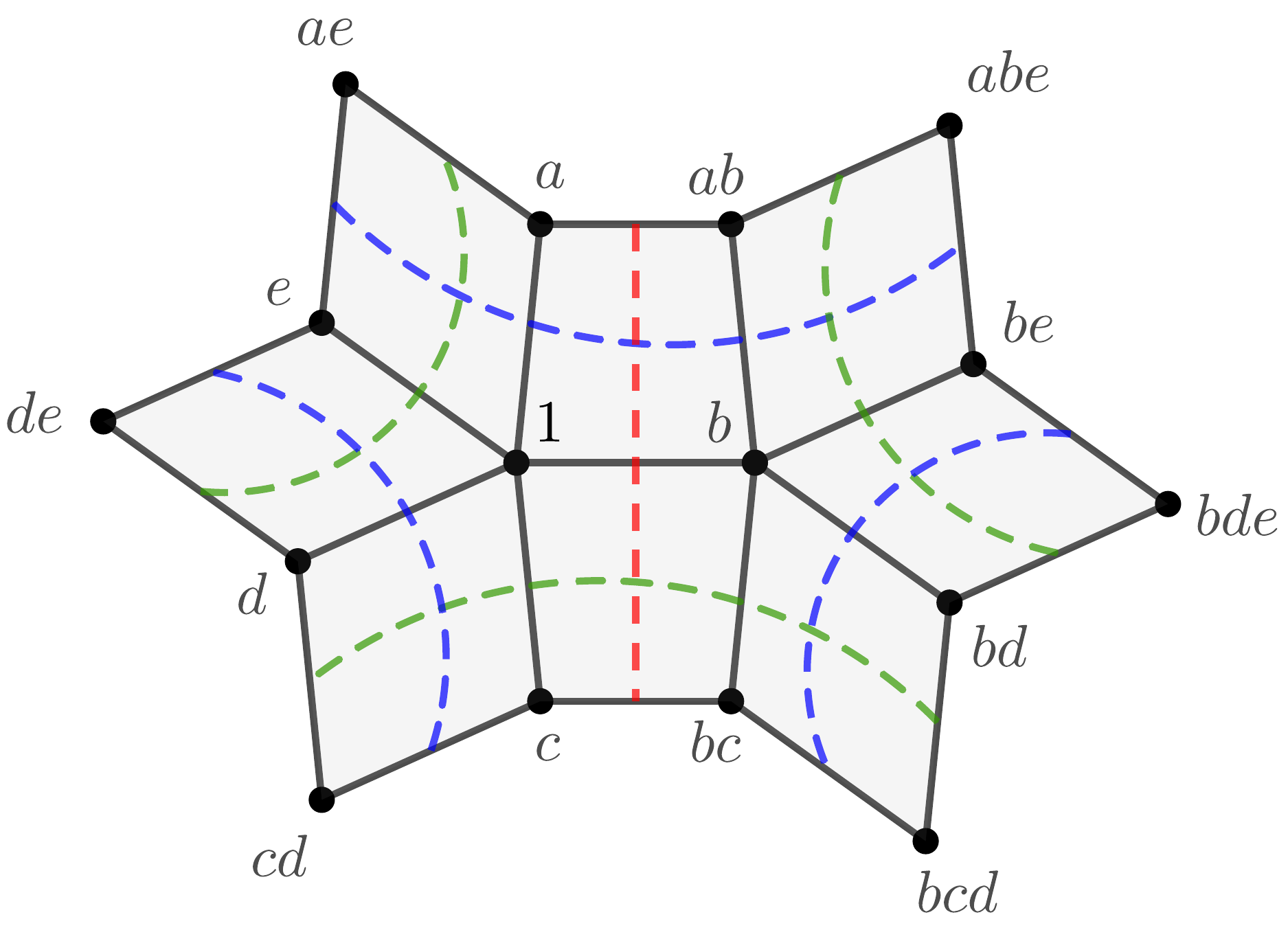}
\end{center}
\noindent
Thus, $\{ c,d,e,a,beb^{-1},bdb^{-1} \}$ is a basis for a subgroup of $A(\Gamma)$ isomorphic to the right-angled Artin group over a $6$-cycle.
\end{ex}

\begin{ex}
Let $\Gamma$ be a $3$-path $(a,b,c,d)$. By rotating the hyperplanes $J_a,J_b,J_c$ around $J_d$, we find a $4$-path of hyperplanes $J_a,J_b,J_c,dJ_b,dJ_a$.
\begin{center}
\includegraphics[width=0.4\linewidth]{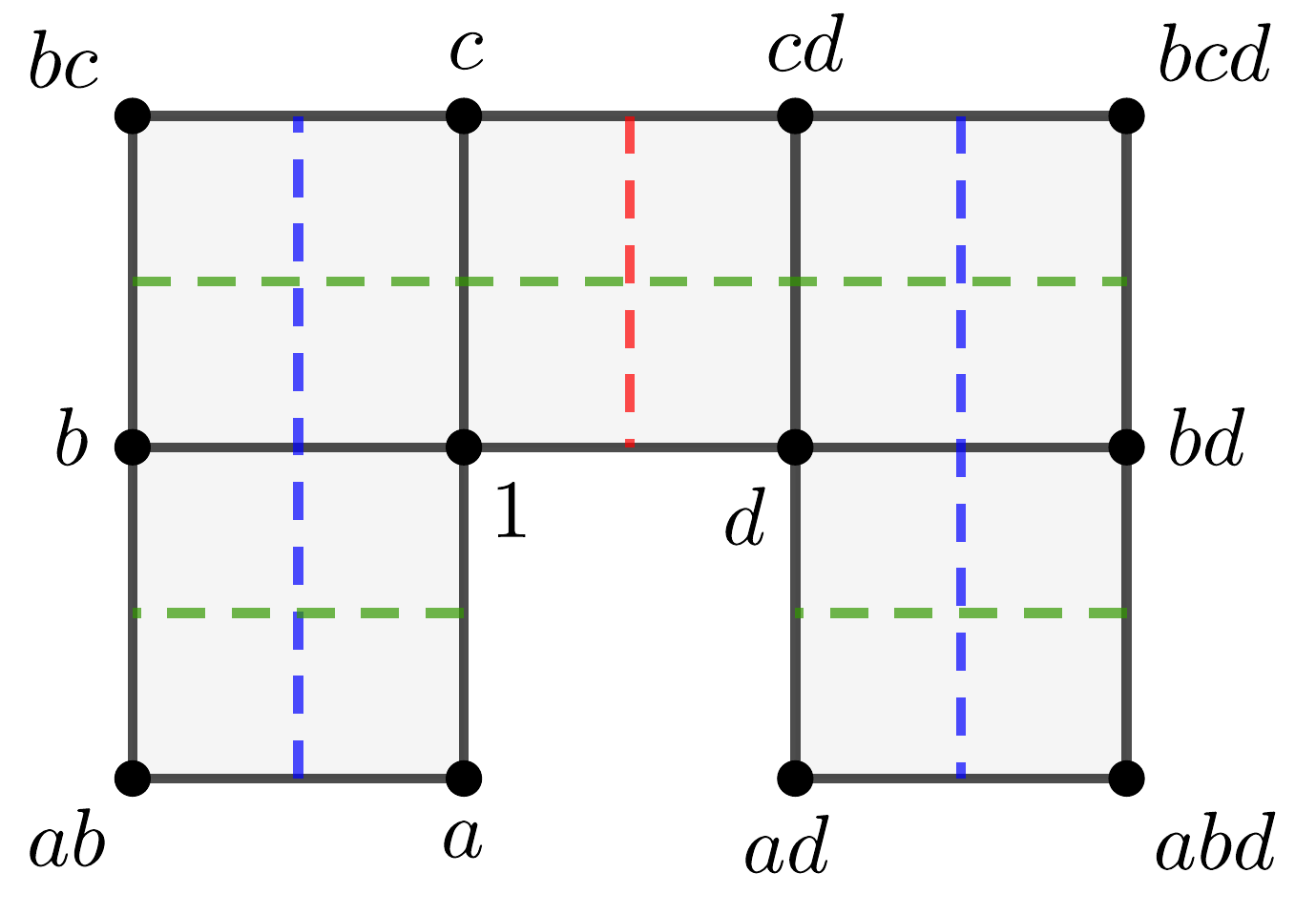}
\end{center}
Thus, $\{a,b,c,dbd^{-1},dad^{-1}\}$ is a basis for a subgroup of $A(\Gamma)$ isomorphic to the right-angled Artin group over a $4$-path.
\end{ex}

\begin{ex}
The two previous examples are particular cases of a general construction. Let $\Gamma$ be an arbitrary graph and $u \in V(\Gamma)$ a vertex. By rotating the hyperplanes $\{J_v \mid v \in V(\Gamma) \backslash \{u\} \}$ around $J_u$ by applying $u$, we get a collection of hyperplanes $\{J_v \mid v \in V(\Gamma) \backslash \{u\} \} \cup \{J_u\} \cup \{uJ_v \mid v \in V(\Gamma) \backslash \mathrm{star}(u) \}$ whose crossing graph coincides with the amalgamation $\Gamma \underset{\mathrm{star}(u)}{\ast} \Gamma$. Consequently, $V(\Gamma) \cup u \left( V(\Gamma) \backslash \mathrm{star}(u) \right) u^{-1}$ is a basis for a subgroup of $A(\Gamma)$ isomorphic to the right-angled Artin group $A \left( \Gamma \underset{\mathrm{star}(u)}{\ast} \Gamma \right)$. Moreover, it can be shown that this subgroup coincides with the kernel of $A(\Gamma) \twoheadrightarrow \langle u \rangle \simeq \mathbb{Z} \twoheadrightarrow \mathbb{Z}/2\mathbb{Z}$. In particular, it has index two in $A(\Gamma)$.  
\end{ex}

\addcontentsline{toc}{section}{Exercices}
\section*{Exercices}

\begin{exercice}
Let $\Gamma$ be a graph. For every $u \in V(\Gamma)$, fix an integer $p(u) \in \mathbb{Z}_{\neq 0}$. Show that the subgroup $\langle u^{p(u)} \mid u \in V(\Gamma) \rangle$ in $A(\Gamma)$ is again isomorphic to $A(\Gamma)$. Give one proof by using the normal form given by Proposition~\ref{prop:NormalForm}, and another one by playing ping-pong on the quasi-median graph $\mathrm{QM}(\Gamma)$.  
\end{exercice}

\begin{exercice}\label{ex:EmbeddingTree}
If $\Gamma$ is a finite forest, show that $A(\Gamma)$ embeds into $A(P_4)$, where $P_4$ denotes a path of length $4$.
\end{exercice}

\begin{exercice}
Let $\Gamma$ be a graph and $\Lambda \subset \Gamma$ an induced subgraph. Let $\rho_\Lambda : A(\Gamma) \to \langle \Lambda \rangle$ denote the retraction that kills all the generators not in $V(\Lambda)$. 
\begin{enumerate}
	\item Show that $\mathrm{ker}(\rho_\Lambda)$ is a right-angled Artin group $A(\Psi)$. (\emph{Hint: Show that $\mathrm{ker}(\rho_\Lambda)$ is generated by the rotative-stabilisers of the hyperplanes tangent to the subgraph $\langle \Lambda \rangle$ in $\mathrm{QM}(\Gamma)$.})
	\item Show that, if $\Gamma$ is finite, then every non-trivial subgroup in $A(\Gamma)$ surjects onto $\mathbb{Z}$. (\emph{Hint: Show the statement in the more general situation where $\Gamma$ is allowed to be infinite but with a finite chromatic number. If $\Lambda$ is a maximal monochromatic subgraph, show that the chromatic number of the $\Psi$ from the previous section is smaller that the chromatic number of $\Gamma$.})
	\item Prove that $\langle a,b \mid aba^{-1}=b^{-1}, bab^{-1}=a^{-1} \rangle$ is torsion free, virtually $\mathbb{Z}^2$, and does not embed in a right-angled Artin group.
\end{enumerate}
\end{exercice}

\begin{exercice}
Let $\Gamma$ be graph.
\begin{enumerate}
	\item Show that, for every induced subgraph $\Lambda \subset \Gamma$, the sugraph $\langle \Lambda \rangle$ is convex in $\mathrm{M}(\Gamma)$.
	\item Fix two induced subgraphs $\Phi,\Psi$ and an element $g \in A(\Gamma)$. Assume that $1$ and $g$ are two vertices minimising the distance between $\langle \Phi \rangle$ and $g \langle \Psi \rangle$. Prove that $$\mathrm{stab}(\langle \Phi \rangle) \cap \mathrm{stab}(g \langle \Psi \rangle) = \langle \Phi \cap \Psi \cap \mathrm{link}(\mathrm{supp}(g)) \rangle.$$ (\emph{Hint: If $h$ stabilises $\langle \Phi \rangle$ and $g \langle \Psi \rangle$, observe that the hyperplanes separating $1$ and $h$ coincide with the hyperplanes separating $g$ and $hg$; and that the hyperplanes separating $1$ and $g$ coincide with the hyperplanes separating $h$ and $hg$.})
	\item Deduce that, in $A(\Gamma)$ the intersection of two parabolic subgroups\footnote{In $A(\Gamma)$, a subgroup is parabolic if it conjugate to $\langle \Lambda \rangle$ for some induced subgraph $\Lambda \subset \Gamma$.} is again a parabolic subgroup.
	\item Given an induced subgraph $\Lambda \subset \Gamma$, show that the normaliser of $\langle \Lambda \rangle$ is $\langle \Lambda \cup \mathrm{link}(\Lambda) \rangle$.
\end{enumerate}
\end{exercice}

\newpage

\section{Lecture 4: An algorithm in dimension two}

\noindent
As already mentioned, the Embedding Problem among right-angled Artin groups is widely open in full generality. Nevertheless, it turns out that the problem can be algorithmically solved under the assumption that the graphs under consideration have girth $>3$ (i.e. do not contain $3$-cycles). 

\begin{thm}\label{thm:Algo}
There exists an algorithm that decides, given two finite graphs $\Phi,\Psi$ of girth $>3$, whether or not $A(\Phi)$ embeds into $A(\Psi)$. If so an explicit basis is provided.
\end{thm}

\noindent
A \emph{basis} of a right-angled Artin group $A$ refers to a generating set $R \subset A$ such that the inclusion map $R \hookrightarrow A$ extends to an isomorphism $A(C_R) \to A$, where $C_R$ denotes the commutation graph of $R$. Theorem~\ref{thm:Algo} is proved in \cite{MR3436157}, based on \cite{MR3039768}. The proof we give follows \cite{MR4359526}. 

\medskip \noindent
The theorem will be an easy consequence of the following statement: 

\begin{prop}\label{prop:IffSubgraph}
Let $\Phi,\Psi$ be two finite graphs of girth $>3$. If $A(\Phi)$ embeds into $A(\Psi)$, then $\mathrm{QM}_{|\Phi|}(\Psi)$ contains a collection $\mathcal{J}$ of hyperplanes intersecting the ball of radius $8|\Psi| |\Phi|$ whose crossing graph is isomorphic to $\Phi$. Moreover, if $\mathcal{J}=\{g_1J_{u_1},\ldots, g_kJ_{u_k}\}$ then $\left\{ g_1u_1^{|\Phi|} g_1^{-1}, \ldots, g_ku_k^{|\Phi|}g_k^{-1} \right\}$ is a basis for a subgroup isomorphic to $A(\Phi)$. 
\end{prop}

\noindent
\begin{minipage}{0.49\linewidth}
\begin{center}
\includegraphics[width=0.9\linewidth]{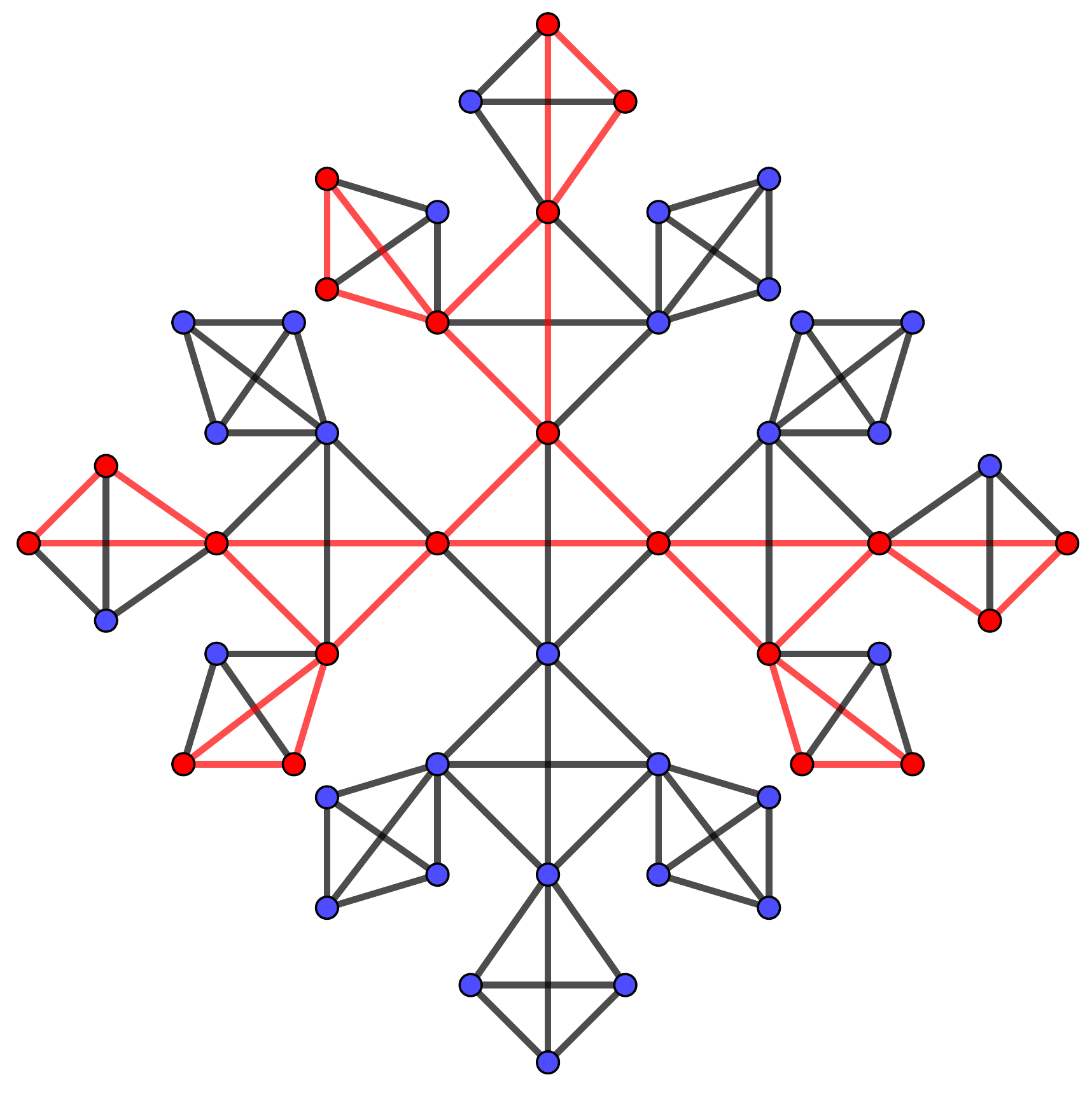}
\end{center}
\end{minipage}
\begin{minipage}{0.5\linewidth}
Here, given an $\ell \geq 1$, $\mathscr{C}_\ell(\Gamma)$ denotes the crossing graph of the subgraph\footnotemark $\mathrm{QM}_\ell(\Gamma)$ of $\mathrm{QM}(\Gamma)$ given by the elements of $A(\Gamma)$ that cannot be represented by graphically reduced words with powers of generators larger than $\ell$. Roughly speaking, $\mathrm{QM}_\ell(\Gamma)$ is a convex subgraph of $\mathrm{QM}(\Gamma)$ obtained by considering only $\ell$ sectors for each hyperplane among the infinitely many sectors it delimits.
\end{minipage}
\footnotetext{It is worth noticing that $\mathrm{QM}_\ell(\Gamma)$ coincides with the Cayley graph of the graph product $\Gamma \mathbb{Z}_\ell$ over $\Gamma$ of cyclic groups of order $\ell$. As a consequence, we deduce that $A(\Phi)$ embeds into $A(\Psi)$ if and only so does $\Phi \mathbb{Z}_?$ into $\Psi \mathbb{Z}_?$.}

\medskip \noindent
Observe that Theorem~\ref{thm:Algo} follows. Indeed, given our two graphs $\Phi$ and $\Psi$, an algorithm goes as follows.
\begin{enumerate}
	\item Draw the crossing graph $\mathscr{C}_{|\Phi|}(\Psi)$ of the hyperplanes in $\mathrm{QM}_{|\Phi|}(\Psi)$ that cross the ball of radius $8|\Psi| |\Phi|$. This is algorithmically possible thanks to Lemma~\ref{lem:QMGamma} and Exercice~\ref{ex:TransverseHyp}.
	\item Look for an induced subgraph isomorphic to $\Phi$ in $\mathscr{C}_{|\Phi|}(\Psi)$. This requires only finitely many operations since our graph is finite.
	\item If there exists such a graph, then Proposition~\ref{prop:IffSubgraph} gives us the basis of a subgroup isomorphic to $A(\Phi)$. Otherwise, $A(\Phi)$ does not embed into $A(\Psi)$. 
\end{enumerate}
It is worth noticing that Proposition~\ref{prop:IffSubgraph}, in particular, proves the converse of Proposition~\ref{cor:KK} for graphs of girth $>3$. Compare with \cite{MR3072113}. This allows us to solve the Embedding Problem for specific families of graphs; see Exercice~\ref{ex:SometimesIFF}. 

\begin{proof}[Proof of Proposition~\ref{prop:IffSubgraph}.]
We begin by showing that there exists a collection of hyperplanes in $\mathrm{QM}(\Psi)$ whose crossing graph is $\Phi$.

\medskip \noindent
Let $\varphi : A(\Phi) \hookrightarrow A(\Psi)$ be an injective morphism and let $\Phi_0$ denote the graph obtained from $\Phi$ by removing the leaves (i.e. the vertex of degree one). For every $u \in \Phi_0$, the centraliser of $u$ contains a non-abelian free subgroup (e.g. the subgroup generated by two neighbours of $u$ in $\Phi$). Consequently, $\varphi(u)$ must be an element whose centraliser is non-abelian. Because $\Psi$ has girth $>3$, it follows from the description of centralisers provided by Theorem~\ref{thm:Centralisers} that $\varphi(u)=g(u) a(u)^{k(u)} g(u)^{-1}$ for some $g(u) \in A(\Psi)$, $a(u) \in \Psi$, and $k(u) \in \mathbb{Z}_{\neq 0}$. More geometrically, this means that $\varphi(u)$ belongs to the rotative-stabiliser of the hyperplane $g(u) J_{a(u)}$. Set $\mathcal{J}_0:= \{ g(u)J_{a(u)} \mid u \in \Phi \}$. We know from Proposition~\ref{prop:SubRAAG} that, for $N$ large enough, $\langle \varphi(u)^N, \ u \in \Phi \rangle$ is isomorphic to $A(\mathscr{C}(\mathcal{J}_0))$. Therefore, $\mathscr{C}(\mathcal{J}_0)$ has to be isomorphic to $\Phi_0$. (More precisely, the map $u \mapsto g(u) J_{a(u)}$ induces an isomorphism $\Phi_0 \to \mathscr{C}(\mathcal{J}_0)$.)

\medskip \noindent
So far, we have found a collection of hyperplanes $\mathcal{J}_0$ whose crossing graph is $\Phi_0$. If $\Phi$ has no leaf, then our initial assertion is proved. In the general case, here is an additional argument.

\medskip \noindent
Let $\ell_1, \ldots, \ell_r$ denote the leaves of $\Phi$. For every $1 \leq i \leq r$, let $u_i$ denote the unique neighbour of $\ell_i$ in $\Phi$ and let $v_i \in \Phi_0$ be a neighbour of $u_i$ distinct from $\ell_i$. We associate to each $\ell_i$ the hyperplane $\varphi(\ell_i) g(v_i)J_{a(v_i)}$. We claim that the crossing graph of the collection $\mathcal{J}$ obtained from $\mathcal{J}_0$ by adding these new hyperplanes is isomorphic to $\Phi$. 
\begin{itemize}
	\item Because $u_i$ and $v_i$ are adjacent in $\Phi_0$, the hyperplanes $g(u_i)J_{a(u_i)}$ and $g(v_i)J_{a(v_i)}$ are transverse. Since $\ell_i$ is adjacent to $u_i$, $\varphi(\ell_i)$ stabilises the hyperplane $g(u_i) J_{a(u_i)}$, which implies that $\varphi(\ell_i)g(v_i)J_{a(v_i)}$ is transverse to $g(u_i)J_{a(u_i)}$.
	\item If $\varphi(\ell_i)g(v_i)J_{a(v_i)}$ is transverse to $g(w)J_{a(w)}$ for some $w \in \Phi_0$, then the rotative-stabilisers have to commute, which implies that $\varphi(\ell_i v_i)$ commute with $\varphi(w)$. The only possibility is that $w=u_i$. 
	\item If $\varphi(\ell_i)g(v_i)J_{a(v_i)}$ is transverse to $\varphi(\ell_j)g(v_j)J_{a(v_j)}$, then the rotative-stabilisers have to commute, which implies that $\varphi(\ell_iv_i)$ commute with $\varphi(\ell_jv_j)$, which can only happen when $i=j$.
\end{itemize}
Thus, we have proved that $\mathscr{C}(\mathcal{J})$ is obtained from $\mathscr{C}(\mathcal{J}_0)$ by adding a leaf to each $g(u_i)J_{a(u_i)}$, which proves that $\mathscr{C}(\mathcal{J})$ is isomorphic to $\Phi$, as desired.

\medskip \noindent
Now, assume that $\mathcal{J}$ contains a hyperplane that does not cross the ball centered at $1$ of radius $8|\Psi| |\Phi|$. In other words, $\mathcal{J}$ contains a hyperplane $J$ that is very far from $1$. This implies that we can find two non-transverse hyperplane $A,B$ in $\mathrm{QM}(\Gamma)$ that separate $1$ from $J$ and that lie in the same $A(\Psi)$-orbit.
\begin{center}
\includegraphics[width=\linewidth]{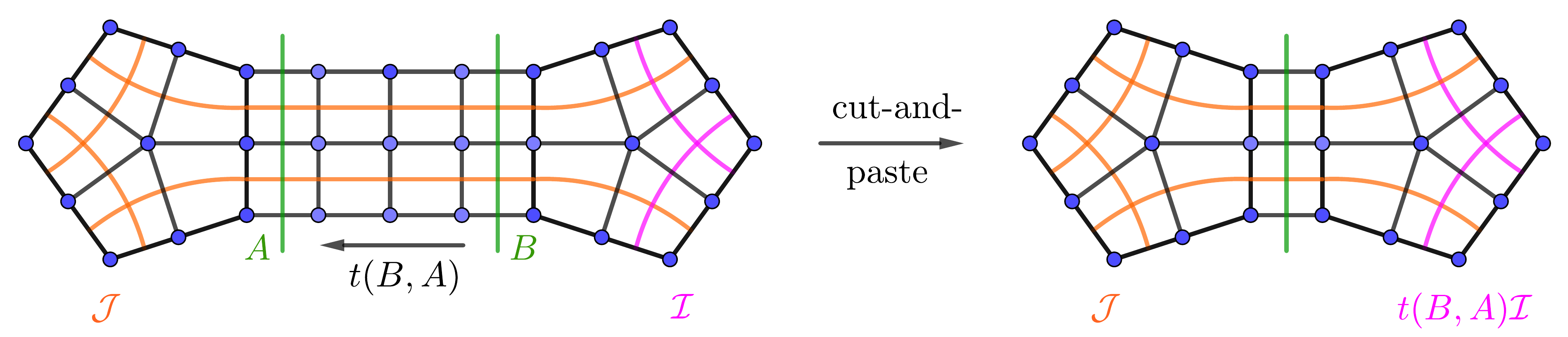}
\end{center}
Let $t(B,A)$ denote the unique element of minimal length in $A(\Psi)$ that sends $B$ to $A$. If $\mathcal{I}$ denotes the subcollection of $\mathcal{J}$ given by the hyperplanes that are separated from $1$ by $B$, consider the new collection $\mathcal{J}':= \mathcal{J} \backslash \mathcal{I} \sqcup t(B,A) \mathcal{I}$. It is obtained from $\mathcal{J}$ by a \emph{cut-and-paste}. We claim that $\mathcal{J}'$ has the same crossing graph as $\mathcal{J}$ and that
$$\sum\limits_{J \in \mathcal{J}'} d_\mathrm{QM}(1,J) < \sum\limits_{J \in \mathcal{J}} d_\mathrm{QM}(1,J).$$
We refer to \cite[Section~5.1]{MR4359526} for the details (written for median graphs and right-angled Coxeter groups, but which can be repeated word for word). By iterating the process, we eventually get a collection of hyperplanes $\mathcal{J}''$ in $\mathrm{QM}(\Gamma)$ whose crossing graph is isomorphic to $\Phi$ and all of whose element intersect the ball of radius $8|\Psi| |\Phi|$. 

\medskip \noindent
For each hyperplane $J$ of $\mathrm{QM}(\Psi)$, select $|\Phi|$ sectors delimited by $J$ in such a way that every hyperplane in $\mathcal{J}$ intersects at least one of these sectors. Let $S(J)$ denote the union of all these sectors delimited by $J$. The intersection 
$$\bigcap\limits_{J \text{ hyperplane}} S(J) \subset \mathrm{QM}(\Psi)$$
defines a convex subgraph isomorphic to $\mathrm{QM}_{|\Phi|}(\Psi)$ in which $\mathcal{J}''$ induces a collection of hyperplanes with $\Phi$ as its crossing graph. 
\end{proof}

\addcontentsline{toc}{section}{Exercices}
\section*{Exercices}

\begin{exercice}\label{ex:TransverseHyp}
Let $\Gamma$ be a finite graph. For all $u,v \in V(\Gamma)$ and $g,h \in A(\Gamma)$, show that the hyperplanes $gJ_u$ and $hJ_v$ are transverse in $\mathrm{QM}(\Gamma)$ if and only if $\{u,v \} \in E(\Gamma)$ and $g^{-1}h \in \langle \mathrm{star}(u) \rangle \cdot \langle \mathrm{star}(v) \rangle$. Deduce an algorithm that determine whether or not two given hyperplanes in $\mathrm{QM}(\Gamma)$ are transverse. 
\end{exercice}

\begin{exercice}\label{ex:SometimesIFF}
For every $k \geq 4$, we denote by $C_k$ the cycle of length $k$ and $P_k$ the path of length $k$.
\begin{enumerate}
	\item Show that $A(C_p)$ embeds in $A(C_q)$ if and only if $q-4$ divides $p-4$.	
	\item Given a finite graph $\Gamma$, show that $A(\Gamma)$ embeds in $A(P_4)$ if and only if $\Gamma$ is a forest. 
	\item Given a graph $\Gamma$ of girth $>3$, show that $\mathbb{F}_2 \times \mathbb{F}_2$ embeds into $A(\Gamma)$ if and only if $\Gamma$ contains an induced $4$-cycle.
\end{enumerate}
\end{exercice}

\newpage

\addcontentsline{toc}{section}{References}

\bibliographystyle{alpha}
{\footnotesize\bibliography{MiniCours}}

\Address

%

\end{document}